\documentclass[12pt,dvipsnames]{amsart}
\usepackage{color}
\usepackage[colorlinks=true, pdfstartview=FitV, linkcolor=blue, citecolor=blue, urlcolor=blue]{hyperref}
\usepackage{geometry}
\usepackage{amsmath, amsthm, amssymb}

\emergencystretch=\maxdimen
\newtheorem{Theorem}{Theorem}[section]

\newtheorem{Corollary}[Theorem]{Corollary}
\newtheorem{Lemma}[Theorem]{Lemma}

\theoremstyle{definition}

\def\Sym{\operatorname{\mathsf{Sym}}}
\def\DSym{\operatorname{\mathsf{DSym}}}
\def\QSym{\operatorname{\mathsf{QSym}}}
\def\DQsym{\operatorname{\mathsf{DQSym}}}

\begin{document}
\title{Ideals and quotients of diagonally quasisymmetric functions}
\author[Shu Xiao Li]{Shu Xiao Li}
\address[Shu Xiao Li]{Department of Mathematics and Statistics\\ York  University\\       To\-ron\-to, Ontario M3J 1P3\\ CANADA} \email{lishu3@yorku.ca}
\date{\today}
 
\begin{abstract}
In 2004, J-C. Aval, F. Bergeron and N. Bergeron studied the algebra of diagonally quasi-symmetric functions $\operatorname{\mathsf{DQSym}}$ in the ring $\mathbb{Q}[\mathbf{x},\mathbf{y}]$ with two sets of variables. They made conjectures on the structure of the quotient $\mathbb{Q}[\mathbf{x},\mathbf{y}]/\langle\operatorname{\mathsf{DQSym}}^+\rangle$, which is a quasi-symmetric analogue of the diagonal harmonic polynomials. In this paper, we construct a Hilbert basis for this quotient when there are infinitely many variables i.e. $\mathbf{x}=x_1,x_2,\dots$ and $\mathbf{y}=y_1,y_2,\dots$. Then we apply this construction to the case where there are finitely many variables, and compute the second column of its Hilbert matrix.
\end{abstract}

\maketitle
\setcounter{tocdepth}{3}

\section{Introduction}

In the polynomial ring $\mathbb{Q}[\mathbf{x}_n]=\mathbb{Q}[x_1,\dots,x_n]$ with $n$ variables, the ring of symmetric polynomials, $\Sym_n$, (cf. \cite{Mac} or \cite{Sagan}) is the subspace of invariants under the symmetric group $S_n$ action
\[
\sigma\cdot f(x_1,x_2,\dots,x_n)=f(x_{\sigma(1)},x_{\sigma(2)},\dots,x_{\sigma(n)}).
\]
The quotient space $\mathbb{Q}[\mathbf{x}_n]/\langle \Sym_n^+\rangle$ over the ideal generated by symmetric polynomials with no constant term is thus called the coinvariant space of symmetric group. Classic results by \cite{Artin} and \cite{Steinberg} asserts that this quotient forms an $S_n$-module that is isomorphic to the left regular representation. Moreover, considering the natural scalar product
\[
\langle f,g\rangle = \big(f(\partial x_1,\dots,\partial x_n)(g(x_1,\dots,x_n))\big)(0,0,\dots, 0),
\]
this quotient is equal to the orthogonal complement of $\Sym_n$. In particular, the coinvariant space is killed by Laplacian operator $\Delta=\partial x_1^2+\cdots+\partial x_n^2$. Hence, it is also known as the harmonic space.

One can show that $\{h_k(x_k,\dots,x_n):1\leq k\leq n\}$ forms a Gr\"{o}bner basis of $\langle \Sym_n^+\rangle$ with respect to the usual order $x_1>\cdots >x_n$, where $h_k$ is the complete homogeneous basis of degree $k$. As a result, the dimension of $\mathbb{Q}[\mathbf{x}_n]/\langle \Sym_n^+\rangle$ is $n!$.

One generalization is the diagonal harmonic space. In the context of $\mathbb{Q}[\mathbf{x}_n,\mathbf{y}_n]=\mathbb{Q}[x_1,\dots,x_n,y_1,\dots,y_n]$, the diagonally symmetric functions, $\DSym_n$, is the space of invariants under the diagonal action of $S_n$
\[
\sigma\cdot f(x_1,\dots,x_n,y_1,\dots,y_n)=f(x_{\sigma(1)},\dots,x_{\sigma(n)},y_{\sigma(1)},\dots,y_{\sigma(n)}).
\]
The diagonal harmonics, $\mathbb{Q}[\mathbf{x}_n,\mathbf{y}_n]/\langle \DSym_n^+\rangle$, was studied in \cite{GH} and \cite{Haiman} where it was used to prove the $n!$ conjecture and Macdonald positivity. In particular, its dimension turns out to be $(n+1)^{n-1}$. More interesting results and applications can be found in \cite{BBGHT}, \cite{BGHT} and \cite{Haglund}.

The ring of quasi-symmetric functions, $\QSym$, was introduced by \cite{Gessel} as generating function for $P$-partitions \cite{Stanley}. It soon shows great importance in algebraic combinatorics e.g. \cite{ABS}, \cite{GKLLRT}. In our context, $\QSym_n$ can be defined as the space of invariants in $\mathbb{Q}[\mathbf{x}_n]$ under the $S_n$-action of Hivert
\[
\sigma\cdot \left(x_{i_1}^{a_1}\cdots x_{i_k}^{a_k}\right)=x_{j_1}^{a_1}\cdots x_{j_k}^{a_k}
\]
where $i_1<\cdots <i_k$, $j_1<\cdots <j_k$ and $\{j_1,\dots,j_k\}=\{\sigma(i_1),\dots,\sigma(i_k)\}$.

In \cite{AB} and \cite{ABB}, the authors studied the quotient $\mathbb{Q}[\mathbf{x}_n]/\langle \QSym_n^+\rangle$ over the ideal generated by quasi-symmetric polynomials with no constant term, which they called the super-covariant space of $S_n$.
Their main result is that a basis of this quotient corresponds to Dyck paths, and the dimension of the quotient space is the $n$-th Catalan number $C_n$.

After that, in \cite{ABB2}, they extended $\QSym$ to diagonal setting, called diagonally quasi-symmetric functions, $\DQsym$. They described a Hopf algebra structure on $\DQsym$, and made a conjecture about the linear structure of $\mathbb{Q}[\mathbf{x}_n,\mathbf{y}_n]/\langle \DQsym_n^+\rangle$.

In this paper, we continue the study of the linear structure. We start with the case where there are infinitely many variables i.e. $R=\mathbb{Q}[[\mathbf{x},\mathbf{y}]]$ is the ring of formal power series where $\mathbf{x}=x_1,x_2,\dots$ and $\mathbf{y}=y_1,y_2,\dots$. The main result is that we give a description of a Hilbert basis for the quotient space $R/I$ where $I=\overline{\DQsym^+}$ is the closure of the ideal generated by $\DQsym$ without constant terms. This Hilbert basis gives an upper bound on the degree of $\mathbb{Q}[\mathbf{x}_n,\mathbf{y}_n]/\langle \DQsym_n^+\rangle$. We then use it to compute the second column of the Hilbert matrix, which coincides with the conjecture in \cite{ABB2}.

\section{Definitions}

\subsection{Bicompositions}
An element $\tilde{\alpha}=\begin{pmatrix}\tilde{\alpha}_{11}&\tilde{\alpha}_{12}&\cdots\\ \tilde{\alpha}_{21}&\tilde{\alpha}_{22}&\cdots\end{pmatrix}\in \mathbb{N}^{2\mathbb{N}}$ is called a generalized bicomposition if all but finitely many $(\tilde{\alpha}_{1k},\tilde{\alpha}_{2k})$ are $(0,0)$. Let $k$ be the maximum number such that $(\tilde{\alpha}_{1k},\tilde{\alpha}_{2k})\neq(0,0)$. The length of $\tilde{\alpha}$, denoted by $\ell(\tilde{\alpha})$, is $k$. The size of $\tilde{\alpha}$, denoted by $|\tilde{\alpha}|$, is the sum of all its entries. For simplicity, we usually write $\tilde{\alpha}$ as $\big({\tilde{\alpha}_{11} \ \cdots \ \tilde{\alpha}_{1k} \atop \tilde{\alpha}_{21} \ \cdots \ \tilde{\alpha}_{2k}} \big)$. There also exists a generalized bicomposition with length $0$ and size $0$, called the zero bicomposition, denoted by $ \big( {0 \atop 0} \big) $.

Every monomial in $R$ can be expressed as $\mathbf{X}^{\tilde{\alpha}}=x_1^{\tilde{\alpha}_{11}}y_1^{\tilde{\alpha}_{21}}\cdots x_k^{\tilde{\alpha}_{1k}} y_k^{\tilde{\alpha}_{2k}}$ for some generalized bicomposition $\tilde{\alpha}$. A generalized bicomposition $\alpha$ is called a bicomposition if $\ell(\alpha)=0$ or $(\alpha_{1j},\alpha_{2j})\neq(0,0)$ for all $1\leq j\leq \ell(\alpha)$.

In this paper, we use Greek letters to denote bicompositions, and Greek letters with tilde to denote generalized bicompositions.

Let $\tilde{\alpha}$, $\tilde{\beta}$ and $\tilde{\gamma}$ be non-zero generalized bicompositions. We write $\tilde{\alpha}=\tilde{\beta}\tilde{\gamma}$ if $\tilde{\alpha}_{ij}=\tilde{\beta}_{ij}$ for all $1\leq j\leq\ell(\tilde{\alpha})-\ell(\tilde{\gamma})$, $\tilde{\beta}_{ij}=0$ for all $j>\ell(\tilde{\alpha})-\ell(\tilde{\gamma})$ and $\tilde{\alpha}_{i(j+\ell(\tilde{\alpha})-\ell(\tilde{\gamma}))}=\tilde{\gamma}_{ij}$ for all $j\geq 1$. We write $\tilde{\alpha}= \big( {0 \atop 0} \big) \tilde{\beta}$ if $\tilde{\alpha}_{11}=\tilde{\alpha}_{21}=0$ and $\tilde{\alpha}_{i(j+1)}=\tilde{\beta}_{ij}$ for all $j\geq2$.

Note that for each generalized bicomposition $\tilde{\alpha}$ that is not a bicomposition, there is a unique way to decompose it into $\tilde{\alpha}=\tilde{\beta} \big( {0 \atop 0} \big) \gamma$ for some generalized bicomposition $\tilde{\beta}$ and bicomposition $\gamma$.

\subsection{Diagonally quasi-symmetric functions}
The algebra of diagonally quasi-symmetric functions, $\DQsym$, is a subalgebra of $\mathbb{Q}[[\mathbf{x},\mathbf{y}]]$ spanned by monomials indexed by bicompositions
$$M_{\alpha}=\sum_{i_1<\dots<i_k}x_{i_1}^{\alpha_{11}}y_{i_1}^{\alpha_{21}}\cdots x_{i_k}^{\alpha_{1k}} y_{i_k}^{\alpha_{2k}}.$$

As a graded algebra, $\displaystyle\DQsym=\bigoplus_{n\geq0}\DQsym_n$ where $\DQsym_n=\text{span-}\{M_{\alpha}: |\alpha|=n\}$ is the degree $n$ component. The algebra structure is defined in \cite{ABB2}.

\subsection{The \texorpdfstring{$F$}{F} basis}

We define a partial ordering $\preceq$ on bicompositions: $\alpha\preceq\beta$ and $\beta$ covers $\alpha$ if there exists a $1\leq k\leq\ell(\alpha)-1$ such that either $\alpha_{2k}=0$ or $\alpha_{1(k+1)}=0$, and $$\beta=\begin{pmatrix}\alpha_{11}&\cdots&\alpha_{1(k-1)}&\alpha_{1k}+\alpha_{1(k+1)}&\alpha_{1(k+2)}&\cdots&\alpha_{1\ell(\alpha)}\\ \alpha_{21}&\cdots&\alpha_{2(k-1)}&\alpha_{2k}+\alpha_{2(k+1)}&\alpha_{2(k+2)}&\cdots&\alpha_{2\ell(\alpha)} \end{pmatrix}.$$ By triangularity, $\displaystyle \left\{F_{\alpha}=\sum_{\alpha\preceq\beta}M_{\beta}\right\}$ forms a basis for $\DQsym$. For example,

$\displaystyle F_{ \big( {2 \atop 2} \big) }=M_{ \big( {2 \atop 2} \big) }+M_{ \big( {2 \ 0 \atop 0 \ 2} \big) }+M_{ \big( {1 \ 1 \atop 0 \ 2} \big) }+M_{ \big( {1 \ 1 \ 0 \atop 0 \ 0 \ 2} \big) }+M_{ \big( {2 \ 0 \atop 1 \ 1} \big) }+M_{ \big( {2 \ 0 \ 0 \atop 0 \ 1 \ 1} \big) }+M_{ \big( {1 \ 1 \ 0 \atop 0 \ 1 \ 1} \big) }+M_{ \big( {1 \ 1 \ 0 \ 0 \atop 0 \ 0 \ 1 \ 1} \big) }.$

For convenience, we set $F_{ \big( {0 \atop 0} \big) }=1$.

This basis has the following easy but important properties

If $\alpha_{11}\geq 1$ and $\alpha_{11}+\alpha_{21}\geq 2$, then
\begin{equation}\tag{2.1}\label{eq:2.1}
F_{\alpha}=x_1 F_{ \big( {\alpha_{11}-1 \ \alpha_{12} \ \cdots \ \alpha_{1\ell(\alpha)} \atop \ \alpha_{21} \ \ \, \alpha_{22} \ \cdots \ \alpha_{2\ell(\alpha)}} \big) }+F_{\alpha}(x_2,x_3,\dots,y_2,y_3,\dots);
\end{equation}
If $\alpha_{11}=1$ and $\alpha_{21}=0$, then
\begin{equation}\tag{2.2}\label{eq:2.2}
F_{\alpha}=x_1 F_{ \big( {\alpha_{12} \ \cdots \ \alpha_{1\ell(\alpha)} \atop \alpha_{22} \ \cdots \ \alpha_{2\ell(\alpha)}} \big) }(x_2,x_3,\dots,y_2,y_3,\dots)+F_{\alpha}(x_2,x_3,\dots,y_2,y_3,\dots);
\end{equation}
If $\alpha_{11}=0$ and $\alpha_{21}\geq 2$, then
\begin{equation}\tag{2.3}\label{eq:2.3}
F_{\alpha}=y_1 F_{ \big( { \ \ 0 \ \ \ \ \alpha_{12} \ \cdots \ \alpha_{1\ell(\alpha)} \atop \alpha_{21}-1 \ \alpha_{22} \ \cdots \ \alpha_{2\ell(\alpha)}} \big) }+F_{\alpha}(x_2,x_3,\dots,y_2,y_3,\dots);
\end{equation}
If $\alpha_{11}=0$ and $\alpha_{21}=1$, then
\begin{equation}\tag{2.4}\label{eq:2.4}
F_{\alpha}=y_1 F_{ \big( {\alpha_{12} \ \cdots \ \alpha_{1\ell(\alpha)} \atop \alpha_{22} \ \cdots \ \alpha_{2\ell(\alpha)}} \big) }(x_2,x_3,\dots,y_2,y_3,\dots)+F_{\alpha}(x_2,x_3,\dots,y_2,y_3,\dots).
\end{equation}

\section{The \texorpdfstring{$G$}{G} basis}

In this section, we define a basis $\{G_{\tilde{\epsilon}}\}$ indexed by generalized bicompositions for $\mathbb{Q}[[\mathbf{x},\mathbf{y}]]$.

Base cases: $G_{\big( {0 \atop 0} \big)}=1$ and $G_{\tilde{\epsilon}}=F_{\tilde{\epsilon}}$ if $\tilde{\epsilon}$ is a bicomposition. Otherwise, let $\tilde{\epsilon}=\tilde{\alpha}\big( {0 \atop 0} \big)\beta$ where $\beta$ is a non-zero bicomposition. Let $k=\ell(\tilde{\epsilon})-\ell(\beta)-1$.

If $\beta_{11}>0$, 
\begin{equation}\tag{3.1}\label{eq:3.1}
G_{\tilde{\epsilon}}=G_{\tilde{\alpha}\beta}-x_{k+1}G_{\tilde{\alpha} \big( {\beta_{11}-1 \ \beta_{12} \ \cdots \ \beta_{1\ell(\beta)} \atop \ \beta_{21} \ \ \, \beta_{22} \ \cdots \ \beta_{2\ell(\beta)}} \big)}.
\end{equation}
If $\beta_{11}=0$,
\begin{equation}\tag{3.2}\label{eq:3.2}
G_{\tilde{\epsilon}}=G_{\tilde{\alpha}\beta}-y_{k+1}G_{\tilde{\alpha} \big( { \ \ 0 \ \ \ \ \beta_{12} \ \cdots \ \beta_{1\ell(\beta)} \atop \beta_{21}-1 \ \beta_{22} \ \cdots \ \beta_{2\ell(\beta)}} \big)}.
\end{equation}
Inductively, $\{G_{\tilde{\epsilon}}\}$ is defined for all generalized bicomposition $\tilde{\epsilon}$. Clearly $G_{\tilde{\epsilon}}$ is homogeneous in degree $|\tilde{\epsilon}|$. Hence, we have a notion of leading monomial of $G_{\tilde{\epsilon}}$, denoted by $LM(G_{\tilde{\epsilon}})$ with respect to the lexicographic order with $x_1>y_1>x_2>y_2>\cdots$. To show that $\{G_{\tilde{\epsilon}}\}$ form a basis, it suffices to prove the leading monomial of $G_{\tilde{\epsilon}}$ is $\mathbf{X}^{\tilde{\epsilon}}$.

\begin{Lemma}\label{lemma:3.1}
Let $\tilde{\alpha}= \big( {a \atop b} \big) \tilde{\beta}$ be a generalized bicomposition,
\begin{enumerate}
\item if $a=b=0$, then $G_{\tilde{\alpha}}=G_{\tilde{\beta}}(x_2,x_3,\dots,y_2,y_3,\dots)$,
\item if $a>0$, then $G_{\tilde{\alpha}}=x_1G_{ \big( {a-1 \atop b} \big) \tilde{\beta}}+P(x_2,x_3,\dots,y_2,y_3,\dots)$,
\item if $a=0$ and $b>0$, then $G_{\tilde{\alpha}}=y_1G_{ \big( {0 \atop b-1} \big) \tilde{\beta}}+P(x_2,x_3,\dots,y_2,y_3,\dots)$
\end{enumerate}
for some $P\in\mathbb{Q}[[\mathbf{x},\mathbf{y}]]$.
\end{Lemma}

\begin{proof}
We prove by induction on the length of $\tilde{\alpha}$.
\begin{enumerate}
\item If $\tilde{\alpha}= \big( {0 \atop 0} \big) $, then $G_{\tilde{\alpha}}=1$ and we are done.
\item If $\tilde{\beta}=\beta$ is a bicomposition,
\begin{enumerate}
\item\label{2a} if $a=b=0$ and $\beta$ non-zero,
\begin{enumerate}
\item if $\beta_{11}\geq1$ and $\beta_{11}+\beta_{21}\geq 2$, using (\ref{eq:2.1}) and (\ref{eq:3.1}), we get
\begin{align*}
G_{\tilde{\alpha}}&=G_{\beta}-x_1G_{ \big( {\beta_{11}-1 \ \beta_{12} \ \cdots \ \beta_{1\ell(\beta)} \atop \ \beta_{21} \ \ \, \beta_{22} \ \cdots \ \beta_{2\ell(\beta)}} \big)}\\
&=F_{\beta}-x_1F_{ \big( {\beta_{11}-1 \ \beta_{12} \ \cdots \ \beta_{1\ell(\beta)} \atop \ \beta_{21} \ \ \, \beta_{22} \ \cdots \ \beta_{2\ell(\beta)}} \big)}\\
&=F_{\beta}(x_2,x_3,\dots,y_2,y_3,\dots)=G_{\beta}(x_2,x_3,\dots,y_2,y_3,\dots)
\end{align*}
and the lemma follows.
\item if $\beta_{11}=1$ and $\beta_{21}=0$, using (\ref{eq:2.2}), (\ref{eq:3.1}) and induction on $\ell(\tilde{\beta})$, we get
\begin{align*}
G_{\tilde{\alpha}}&=G_{\beta}-x_1G_{ \big( {0 \ \beta_{12} \ \cdots \ \beta_{1\ell(\beta)} \atop 0 \ \beta_{22} \ \cdots \ \beta_{2\ell(\beta)}} \big)}\\
&=G_{\beta}-x_1G_{\big( {\beta_{12} \ \cdots \ \beta_{1\ell(\beta)} \atop \beta_{22} \ \cdots \ \beta_{2\ell(\beta)}} \big)}(x_2,x_3,\dots,y_2,y_3,\dots)\\
&=F_{\beta}-x_1F_{\big( {\beta_{12} \ \cdots \ \beta_{1\ell(\beta)} \atop \beta_{22} \ \cdots \ \beta_{2\ell(\beta)}} \big)}(x_2,x_3,\dots,y_2,y_3,\dots)\\
&=F_{\beta}(x_2,x_3,\dots,y_2,y_3,\dots)=G_{\beta}(x_2,x_3,\dots,y_2,y_3,\dots)
\end{align*}
and the lemma follows.
\item if $\beta_{11}=0$ and $\beta_{21}\geq 2$, using (\ref{eq:2.3}) and (\ref{eq:3.2}), we get
\begin{align*}
G_{\tilde{\alpha}}&=G_{\beta}-y_1G_{ \big( { \ \ 0 \ \ \ \ \beta_{12} \ \cdots \ \beta_{1\ell(\beta)} \atop \beta_{21}-1 \ \beta_{22} \ \cdots \ \beta_{2\ell(\beta)}} \big)}\\
&=F_{\beta}-y_1F_{ \big( { \ \ 0 \ \ \ \ \beta_{12} \ \cdots \ \beta_{1\ell(\beta)} \atop \beta_{21}-1 \ \beta_{22} \ \cdots \ \beta_{2\ell(\beta)}} \big)}\\
&=F_{\beta}(x_2,x_3,\dots,y_2,y_3,\dots)=G_{\beta}(x_2,x_3,\dots,y_2,y_3,\dots)
\end{align*}
and the lemma follows.
\item if $\beta_{11}=0$ and $\beta_{21}=1$, using (\ref{eq:2.4}), (\ref{eq:3.2}) and induction on $\ell(\tilde{\beta})$, we get
\begin{align*}
G_{\tilde{\alpha}}&=G_{\beta}-y_1G_{ \big( {0 \ \beta_{12} \ \cdots \ \beta_{1\ell(\beta)} \atop 0 \ \beta_{22} \ \cdots \ \beta_{2\ell(\beta)}} \big)}\\
&=G_{\beta}-y_1G_{\big( {\beta_{12} \ \cdots \ \beta_{1\ell(\beta)} \atop \beta_{22} \ \cdots \ \beta_{2\ell(\beta)}} \big)}(x_2,x_3,\dots,y_2,y_3,\dots)\\
&=F_{\beta}-y_1F_{\big( {\beta_{12} \ \cdots \ \beta_{1\ell(\beta)} \atop \beta_{22} \ \cdots \ \beta_{2\ell(\beta)}} \big)}(x_2,x_3,\dots,y_2,y_3,\dots)\\
&=F_{\beta}(x_2,x_3,\dots,y_2,y_3,\dots)=G_{\beta}(x_2,x_3,\dots,y_2,y_3,\dots)
\end{align*}
and the lemma follows.
\end{enumerate}
\item if $a\geq1$ and $a+b\geq 2$, by definition $G_{\tilde{\alpha}}=F_{\tilde{\alpha}}$. Using (\ref{eq:2.1}), we get
$$G_{\tilde{\alpha}}=F_{\tilde{\alpha}}=x_1F_{ \big( {a-1 \atop b} \big) \beta}+F_{\tilde{\alpha}}(x_2,x_3,\dots,y_2,y_3,\dots)$$
and the lemma follows, with $P=F_{\tilde{\alpha}}(x_2,x_3,\dots,y_2,y_3,\dots)$.
\item if $a=1$ and $b=0$, by definition $G_{\tilde{\alpha}}=F_{\tilde{\alpha}}$. Using (\ref{eq:2.2}) and (\ref{2a}). we get
\begin{align*}
G_{\tilde{\alpha}}=F_{\tilde{\alpha}}&=x_1F_{\beta}(x_2,x_3,\dots,y_2,y_3,\dots)+F_{\tilde{\alpha}}(x_2,x_3,\dots,y_2,y_3,\dots)\\
&=x_1G_{ \big( {0 \atop 0} \big) \beta}+F_{\tilde{\alpha}}(x_2,x_3,\dots,y_2,y_3,\dots)
\end{align*}
and the lemma follows with $P=F_{\tilde{\alpha}}(x_2,x_3,\dots,y_2,y_3,\dots)$.
\item if $a=0$ and $b\geq2$, by definition $G_{\tilde{\alpha}}=F_{\tilde{\alpha}}$. Using (\ref{eq:2.3}), we get
$$G_{\tilde{\alpha}}=F_{\tilde{\alpha}}=y_1F_{ \big( {a \atop b-1} \big) \beta}+F_{\tilde{\alpha}}(x_2,x_3,\dots,y_2,y_3,\dots)$$
and the lemma follows, with $P=F_{\tilde{\alpha}}(x_2,x_3,\dots,y_2,y_3,\dots)$.
\item if $a=0$ and $b=1$, by definition $G_{\tilde{\alpha}}=F_{\tilde{\alpha}}$. Using (\ref{eq:2.4}) and (\ref{2a}). we get
\begin{align*}
G_{\tilde{\alpha}}=F_{\tilde{\alpha}}&=y_1F_{\beta}(x_2,x_3,\dots,y_2,y_3,\dots)+F_{\tilde{\alpha}}(x_2,x_3,\dots,y_2,y_3,\dots)\\
&=y_1G_{ \big( {0 \atop 0} \big) \beta}+F_{\tilde{\alpha}}(x_2,x_3,\dots,y_2,y_3,\dots)
\end{align*}
and the lemma follows with $P=F_{\tilde{\alpha}}(x_2,x_3,\dots,y_2,y_3,\dots)$.
\end{enumerate}
\item In the general case, let $\tilde{\alpha}=\tilde{\gamma} \big( {0 \atop 0} \big) \beta$ where $\beta$ is a non-empty bicomposition and $k=\ell(\tilde{\alpha})-\ell(\beta)-1$. We prove by induction on $k$. If $k=1$, then we are back in case (\ref{2a}) above. Hence, we assume $k>1$ and $\tilde{\gamma}= \big( {a \atop b} \big) \tilde{\mu}$.
\begin{enumerate}
\item If $a=b=0$,
\begin{enumerate}
\item if $\beta_{11}\geq1$, by induction and (\ref{eq:3.1}), we have
\begin{align*}
G_{\tilde{\alpha}}=&G_{ \big( {0 \atop 0} \big) \tilde{\mu} \big( {0 \atop 0} \big) \beta}=G_{ \big( {0 \atop 0} \big) \tilde{\mu}\beta}-x_kG_{ \big( {0 \atop 0} \big) \tilde{\mu} \big( {\beta_{11}-1 \ \beta_{12} \ \cdots \ \beta_{1\ell(\beta)} \atop \ \beta_{21} \ \ \, \beta_{22} \ \cdots \ \beta_{2\ell(\beta)}} \big)}\\
=&G_{\tilde{\mu}\beta}(x_2,x_3,\dots,y_2,y_3,\dots)\\
&-x_{(k-1)+1}G_{\tilde{\mu} \big( {\beta_{11}-1 \ \beta_{12} \ \cdots \ \beta_{1\ell(\beta)} \atop \ \beta_{21} \ \ \, \beta_{22} \ \cdots \ \beta_{2\ell(\beta)}} \big)}(x_2,x_3,\dots,y_2,y_3,\dots)\\
=&G_{\tilde{\mu} \big( {0 \atop 0} \big) \beta}(x_2,x_3,\dots,y_2,y_3,\dots)
\end{align*}
and the lemma follows.
\item if $\beta_{11}=0$, by induction and (\ref{eq:3.2}), we have
\begin{align*}
G_{\tilde{\alpha}}=&G_{ \big( {0 \atop 0} \big) \tilde{\mu} \big( {0 \atop 0} \big) \beta}=G_{ \big( {0 \atop 0} \big) \tilde{\mu}\beta}-y_kG_{ \big( {0 \atop 0} \big) \tilde{\mu} \big( { \ \ 0 \ \ \ \ \beta_{12} \ \cdots \ \beta_{1\ell(\beta)} \atop \beta_{21}-1 \ \beta_{22} \ \cdots \ \beta_{2\ell(\beta)}} \big)}\\
=&G_{\tilde{\mu}\beta}(x_2,x_3,\dots,y_2,y_3,\dots)\\
&-y_{(k-1)+1}G_{\tilde{\mu} \big( { \ \ 0 \ \ \ \ \beta_{12} \ \cdots \ \beta_{1\ell(\beta)} \atop \beta_{21}-1 \ \beta_{22} \ \cdots \ \beta_{2\ell(\beta)}} \big)}(x_2,x_3,\dots,y_2,y_3,\dots)\\
=&G_{\tilde{\mu} \big( {0 \atop 0} \big) \beta}(x_2,x_3,\dots,y_2,y_3,\dots)
\end{align*}
and the lemma follows.
\end{enumerate}
\item If $a\geq1$,
\begin{enumerate}
\item if $\beta_{11}\geq1$, by induction and (\ref{eq:3.1}), we have
\begin{align*}
G_{\tilde{\alpha}}=&G_{ \big( {a \atop b} \big) \tilde{\mu} \big( {0 \atop 0} \big) \beta}=G_{ \big( {a \atop b} \big) \tilde{\mu}\beta}-x_kG_{ \big( {a \atop b} \big) \tilde{\mu} \big( {\beta_{11}-1 \ \beta_{12} \ \cdots \ \beta_{1\ell(\beta)} \atop \ \beta_{21} \ \ \, \beta_{22} \ \cdots \ \beta_{2\ell(\beta)}} \big)}\\
=&x_1G_{ \big( {a-1 \atop b} \big) \tilde{\mu}\beta}+P_1(x_2,x_3,\dots,y_2,y_3,\dots)\\
&-x_k\Big(x_1G_{ \big( {a-1 \atop b} \big) \tilde{\mu} \big( {\beta_{11}-1 \ \beta_{12} \ \cdots \ \beta_{1\ell(\beta)} \atop \ \beta_{21} \ \ \, \beta_{22} \ \cdots \ \beta_{2\ell(\beta)}} \big)}\\
&+P_2(x_2,x_3,\dots,y_2,y_3,\dots)\Big)\\
=&x_1\left(G_{ \big( {a-1 \atop b} \big) \tilde{\mu}\beta}-x_kG_{ \big( {a-1 \atop b} \big) \tilde{\mu} \big( {\beta_{11}-1 \ \beta_{12} \ \cdots \ \beta_{1\ell(\beta)} \atop \ \beta_{21} \ \ \, \beta_{22} \ \cdots \ \beta_{2\ell(\beta)}} \big)}\right)\\
&+P(x_2,x_3,\dots,y_2,y_3,\dots)\\
=&x_1G_{ \big( {a-1 \atop b} \big) \tilde{\mu} \big( {0 \atop 0} \big) \beta}+P(x_2,x_3,\dots,y_2,y_3,\dots)
\end{align*}
and the lemma follows with $P=P_1-x_kP_2$.
\item if $\beta_{11}=0$, by induction and (\ref{eq:3.2}), we have
\begin{align*}
G_{\tilde{\alpha}}=&G_{ \big( {a \atop b} \big) \tilde{\mu} \big( {0 \atop 0} \big) \beta}=G_{ \big( {a \atop b} \big) \tilde{\mu}\beta}-y_kG_{ \big( {a \atop b} \big) \tilde{\mu} \big( { \ \ 0 \ \ \ \ \beta_{12} \ \cdots \ \beta_{1\ell(\beta)} \atop \beta_{21}-1 \ \beta_{22} \ \cdots \ \beta_{2\ell(\beta)}} \big)}\\
=&x_1G_{ \big( {a-1 \atop b} \big) \tilde{\mu}\beta}+P_1(x_2,x_3,\dots,y_2,y_3,\dots)\\
&-y_k\Big(x_1G_{ \big( {a-1 \atop b} \big) \tilde{\mu} \big( { \ \ 0 \ \ \ \ \beta_{12} \ \cdots \ \beta_{1\ell(\beta)} \atop \beta_{21}-1 \ \beta_{22} \ \cdots \ \beta_{2\ell(\beta)}} \big)}\\
&+P_2(x_2,x_3,\dots,y_2,y_3,\dots)\Big)\\
=&x_1\left(G_{ \big( {a-1 \atop b} \big) \tilde{\mu}\beta}-y_kG_{ \big( {a-1 \atop b} \big) \tilde{\mu} \big( { \ \ 0 \ \ \ \ \beta_{12} \ \cdots \ \beta_{1\ell(\beta)} \atop \beta_{21}-1 \ \beta_{22} \ \cdots \ \beta_{2\ell(\beta)}} \big)}\right)\\
&+P(x_2,x_3,\dots,y_2,y_3,\dots)\\
=&x_1G_{ \big( {a-1 \atop b} \big) \tilde{\mu} \big( {0 \atop 0} \big) \beta}+P(x_2,x_3,\dots,y_2,y_3,\dots)
\end{align*}
and the lemma follows with $P=P_1-y_kP_2$.
\end{enumerate}
\item If $a=0$ and $b\geq1$,
\begin{enumerate}
\item if $\beta_{11}\geq1$, by induction and (\ref{eq:3.1}), we have
\begin{align*}
G_{\tilde{\alpha}}=&G_{ \big( {0 \atop b} \big) \tilde{\mu} \big( {0 \atop 0} \big) \beta}=G_{ \big( {0 \atop b} \big) \tilde{\mu}\beta}-x_kG_{ \big( {0 \atop b} \big) \tilde{\mu} \big( {\beta_{11}-1 \ \beta_{12} \ \cdots \ \beta_{1\ell(\beta)} \atop \ \beta_{21} \ \ \, \beta_{22} \ \cdots \ \beta_{2\ell(\beta)}} \big)}\\
=&y_1G_{ \big( {0 \atop b-1} \big) \tilde{\mu}\beta}+P_1(x_2,x_3,\dots,y_2,y_3,\dots)\\
&-x_k\Big(y_1G_{ \big( {0 \atop b-1} \big) \tilde{\mu} \big( {\beta_{11}-1 \ \beta_{12} \ \cdots \ \beta_{1\ell(\beta)} \atop \ \beta_{21} \ \ \, \beta_{22} \ \cdots \ \beta_{2\ell(\beta)}} \big)}\\
&+P_2(x_2,x_3,\dots,y_2,y_3,\dots)\Big)\\
=&y_1\left(G_{ \big( {0 \atop b-1} \big) \tilde{\mu}\beta}-x_kG_{ \big( {0 \atop b-1} \big) \tilde{\mu} \big( {\beta_{11}-1 \ \beta_{12} \ \cdots \ \beta_{1\ell(\beta)} \atop \ \beta_{21} \ \ \, \beta_{22} \ \cdots \ \beta_{2\ell(\beta)}} \big)}\right)\\
&+P(x_2,x_3,\dots,y_2,y_3,\dots)\\
=&y_1G_{ \big( {0 \atop b-1} \big) \tilde{\mu} \big( {0 \atop 0} \big) \beta}+P(x_2,x_3,\dots,y_2,y_3,\dots)
\end{align*}
and the lemma follows with $P=P_1-x_kP_2$.
\item if $\beta_{11}=0$, by induction and (\ref{eq:3.2}), we have
\begin{align*}
G_{\tilde{\alpha}}=&G_{ \big( {0 \atop b} \big) \tilde{\mu} \big( {0 \atop 0} \big) \beta}=G_{ \big( {0 \atop b} \big) \tilde{\mu}\beta}-y_kG_{ \big( {0 \atop b} \big) \tilde{\mu} \big( { \ \ 0 \ \ \ \ \beta_{12} \ \cdots \ \beta_{1\ell(\beta)} \atop \beta_{21}-1 \ \beta_{22} \ \cdots \ \beta_{2\ell(\beta)}} \big)}\\
=&y_1G_{ \big( {0 \atop b-1} \big) \tilde{\mu}\beta}+P_1(x_2,x_3,\dots,y_2,y_3,\dots)\\
&-y_k\Big(y_1G_{ \big( {0 \atop b-1} \big) \tilde{\mu} \big( { \ \ 0 \ \ \ \ \beta_{12} \ \cdots \ \beta_{1\ell(\beta)} \atop \beta_{21}-1 \ \beta_{22} \ \cdots \ \beta_{2\ell(\beta)}} \big)}\\
&+P_2(x_2,x_3,\dots,y_2,y_3,\dots)\Big)\\
=&y_1\left(G_{ \big( {0 \atop b-1} \big) \tilde{\mu}\beta}-y_kG_{ \big( {0 \atop b-1} \big) \tilde{\mu} \big( { \ \ 0 \ \ \ \ \beta_{12} \ \cdots \ \beta_{1\ell(\beta)} \atop \beta_{21}-1 \ \beta_{22} \ \cdots \ \beta_{2\ell(\beta)}} \big)}\right)\\
&+P(x_2,x_3,\dots,y_2,y_3,\dots)\\
=&y_1G_{ \big( {0 \atop b-1} \big) \tilde{\mu} \big( {0 \atop 0} \big) \beta}+P(x_2,x_3,\dots,y_2,y_3,\dots)
\end{align*}
and the lemma follows with $P=P_1-y_kP_2$.
\end{enumerate}
\end{enumerate}
\end{enumerate}
\end{proof}

\begin{Corollary}
Let $\tilde{\epsilon}$ be a generalized bicomposition, then the leading monomial of $G_{\tilde{\epsilon}}$ is $\mathbf{X}^{\tilde{\epsilon}}$. Hence, $\{G_{\tilde{\alpha}}\}$ forms a Hilbert basis for $R$.
\end{Corollary}

\begin{proof}
We prove by induction on $\ell(\tilde{\epsilon})$ and $|\tilde{\epsilon}|$. If $\tilde{\epsilon}=\big( {0 \atop 0} \big)$, by definition $G_{\tilde{\epsilon}}=1=X^{\tilde{\epsilon}}$. Otherwise, let $\tilde{\epsilon}=\big( {a \atop b} \big)\tilde{\beta}$.
\begin{enumerate}
\item If $a=b=0$ and $\tilde{\beta}$ non-zero, by induction on $\ell(\tilde{\epsilon})$ and Lemma \ref{lemma:3.1}, we have
$$LM(G_{\tilde{\epsilon}})=LM(G_{\tilde{\beta}}(x_2,x_3,\dots,y_2,y_3,\dots))=(x_2,x_3,\dots,y_2,y_3,\dots)^{\tilde{\beta}}=\mathbf{X}^{\tilde{\epsilon}}.$$
\item If $a\geq1$, by induction on $|\tilde{\epsilon}|$ and Lemma \ref{lemma:3.1}, we have
$$LM(G_{\tilde{\epsilon}})=LM\left(x_1G_{\big( {a-1 \atop b} \big)\tilde{\beta}}\right)=\mathbf{X}^{\tilde{\epsilon}}.$$
\item If $a=0$ and $b\geq1$, by induction on $|\tilde{\epsilon}|$ and Lemma \ref{lemma:3.1}, we have
$$LM(G_{\tilde{\epsilon}})=LM\left(y_1G_{\big( {0 \atop b-1} \big)\tilde{\beta}}\right)=\mathbf{X}^{\tilde{\epsilon}}.$$
\end{enumerate}

\end{proof}

\section{The Hilbert Basis}

The set $\{x^{\tilde{\alpha}}F_{\beta}\}$ is a spanning set of the ideal $I$. For each $\tilde{\alpha}$ and $\beta$, we write $x^{\tilde{\alpha}}F_{\beta}$ in terms of the $G$ basis by the following rules.

(1) We reorder the product $x^{\tilde{\alpha}}F_{\beta}$ as $\cdots\left( x_2^{\tilde{\alpha}_{21}}\left( y_2^{\tilde{\alpha}_{22}}\left( x_1^{\tilde{\alpha}_{11}}\left( y_1^{\tilde{\alpha}_{21}}F_{\beta}\right) \right) \right) \right) $.

(2) We reduce the above product recursively using (\ref{eq:3.1})
\begin{equation}\tag{4.1}
x_iG_{\tilde{\gamma}}=x_iG_{ \big( {\cdots \ \tilde{\gamma}_{1i} \ \cdots \atop \cdots \ \tilde{\gamma}_{2i} \ \cdots}\big)}=G_{ \big( {\cdots \ \tilde{\gamma}_{1i}+1 \ \cdots \atop \cdots \ \ \tilde{\gamma}_{2i} \ \ \cdots}\big)}-G_{ \big( {\cdots \ 0 \ \tilde{\gamma}_{1i}+1 \ \cdots \atop \cdots \ 0 \ \ \tilde{\gamma}_{2i} \ \ \cdots}\big)};
\end{equation}

or using (\ref{eq:3.2}) when $\tilde{\gamma}_{1i}=0$ for some i,
\begin{equation}\tag{4.2}
y_iG_{\tilde{\gamma}}=y_iG_{ \big( {\cdots \ \tilde{\gamma}_{1i} \ \cdots \atop \cdots \ \tilde{\gamma}_{2i} \ \cdots}\big)}=G_{ \big( {\cdots \ \ \tilde{\gamma}_{1i} \ \ \cdots \atop \cdots \ \tilde{\gamma}_{2i}+1 \ \cdots}\big)}-G_{ \big( {\cdots \ 0 \ \ \tilde{\gamma}_{1i} \ \ \cdots \atop \cdots \ 0 \ \tilde{\gamma}_{2i}+1 \ \cdots}\big)}.
\end{equation}

(3) When $\tilde{\gamma}_{1i}=a>0$, we reduce $y_iG_{\tilde{\gamma}}$ as
\begin{align*}\tag{4.3}
y_1G_{\tilde{\gamma}}&=y_iG_{ \big( {\cdots \ \ a \ \ \cdots \atop \cdots \ \tilde{\gamma}_{2i} \ \cdots}\big)}=y_i\left( G_{ \big( {\cdots \ 0 \ \ a \ \ \cdots \atop \cdots \ 0 \ \tilde{\gamma}_{2i} \ \cdots}\big)} + x_iG_{ \big( {\cdots \ a-1 \ \cdots \atop \cdots \ \tilde{\gamma}_{2i} \ \cdots}\big)} \right) \\
&=y_iG_{ \big( {\cdots \ 0 \ \ a \ \ \cdots \atop \cdots \ 0 \ \tilde{\gamma}_{2i} \ \cdots}\big)} + x_i\left( y_iG_{ \big( {\cdots \ a-1 \ \cdots \atop \cdots \ \tilde{\gamma}_{2i} \ \cdots}\big)}\right) =\cdots\\
&=\sum_{k=0}^{a-1}x_i^k\left( y_iG_{ \big( {\cdots \ 0 \ a-k \ \cdots \atop \cdots \ 0 \ \tilde{\gamma}_{2i} \ \cdots}\big)}\right) + x_i^a\left( y_i G_{ \big( {\cdots \ \ 0 \ \ \cdots \atop \cdots \ \tilde{\gamma}_{2i} \ \cdots}\big)}\right).
\end{align*}

The $``\cdots"$ above means $\tilde{\gamma}_{11} \ \cdots \ \tilde{\gamma}_{1(i-1)}$, $\tilde{\gamma}_{1(i+1)} \ \cdots \ \tilde{\gamma}_{1\ell(\tilde{\gamma})}$, $\tilde{\gamma}_{21} \ \cdots \ \tilde{\gamma}_{2(i-1)}$ or

$\tilde{\gamma}_{2(i+1)} \ \cdots \ \tilde{\gamma}_{1\ell(\tilde{\gamma})}$ with respect to their positions in the generalized bicomposition.

For example,
\begin{align*}
&y_1F_{\big( {1 \atop 0} \big)}=y_1\left( G_{\big( {0 \ 1 \atop 0 \ 0} \big)} + x_1G_{\big( {0 \atop 0} \big)}\right) = y_1G_{\big( {0 \ 1 \atop 0 \ 0} \big)}+x_1y_1G_{\big( {0 \atop 0} \big)}\\
=&G_{\big( {0 \ 1 \atop 1 \ 0} \big)} - G_{\big( {0 \ 0 \ 1 \atop 0 \ 1 \ 0} \big)}+x_1\left( G_{\big( {0 \atop 1} \big)}-G_{\big( {0 \ 0 \atop 0 \ 1} \big)}\right)\\
=&G_{\big( {0 \ 1 \atop 1 \ 0} \big)} - G_{\big( {0 \ 0 \ 1 \atop 0 \ 1 \ 0} \big)}+G_{\big( {1 \atop 1} \big)} - G_{\big( {0 \ 1 \atop 0 \ 1} \big)} - G_{\big( {1 \ 0 \atop 0 \ 1} \big)} + G_{\big( {0 \ 1 \ 0 \atop 0 \ 0 \ 1} \big)}.
\end{align*}

For each of the above rule, we choose one $G_{\tilde{\eta}}$ as leading basis element. We define a function $\phi$ from $\left(\{x_i\}\times\{G_{\tilde{\gamma}}\}\right)\cup\left(\{y_i\}\times\{G_{\tilde{\gamma}}\}\right)$ to $\{G_{\tilde{\gamma}}\}$ as follows. In the case of rules (4.1), (4.2), we choose $\phi\left( x_i,G_{\tilde{\gamma}}\right) =G_{ \big( {\cdots \ 0 \ \tilde{\gamma}_{1i}+1 \ \cdots \atop \cdots \ 0 \ \ \tilde{\gamma}_{2i} \ \ \cdots}\big)}$ and $\phi\left( y_i,G_{\tilde{\gamma}}\right) =G_{ \big( {\cdots \ 0 \ \ \tilde{\gamma}_{1i} \ \ \cdots \atop \cdots \ 0 \ \tilde{\gamma}_{2i}+1 \ \cdots}\big)}$. In the case of rule (4.3), we choose $\phi\left( y_i,G_{\tilde{\gamma}}\right) =\phi\left( y_i,G_{ \big( {\cdots \ 0 \ \ a \ \ \cdots \atop \cdots \ 0 \ \tilde{\gamma}_{2i} \ \cdots}\big)}\right) =G_{ \big( {\cdots \ 0 \ 0 \ \ a \ \ \cdots \atop \cdots \ 0 \ 1 \ \tilde{\gamma}_{2i} \ \cdots}\big)}$. In the other words, at each step of the expansion, we choose the lexicographically smallest $\tilde{\eta}$ such that $G_{\tilde{\eta}}$ appears as a term in the expansion.

\begin{Lemma}
The above process of choosing is invertible, i.e. $\phi$ is injective.
\end{Lemma}

\begin{proof}
Since each time we multiply $x_i$ or $y_i$, the chosen term contains a $\binom00$ at position $i$. Combining this fact with the rule that we have to multiply $y_i$ before $x_i$, we have the following inverse function.

Let $i$ be the largest number that $\left( \tilde{\gamma}_{1i},\tilde{\gamma}_{2i}\right) =(0,0)$ and $0<i<\ell(\tilde{\gamma})$.

(1) If $\tilde{\gamma}_{1(i+1)}>0$, then $\phi^{-1}\left( G_{ \big( {\cdots \ 0 \ \tilde{\gamma}_{1(i+1)} \ \cdots \atop \cdots \ 0 \ \tilde{\gamma}_{2(i+1)} \ \cdots}\big)}\right) =x_iG_{ \big( {\cdots \ \tilde{\gamma}_{1(i+1)}-1 \ \cdots \atop \cdots \ \ \tilde{\gamma}_{2(i+1)} \ \ \cdots}\big)}$.

(2) If $\tilde{\gamma}_{1(i+1)}=0$ and, $\tilde{\gamma}_{1(i+2)}=0$ or $\tilde{\gamma}_{2(i+1)}>1$, then

$\phi^{-1}\left( G_{ \big( {\cdots \ 0 \ \tilde{\gamma}_{1(i+1)} \ \cdots \atop \cdots \ 0 \ \tilde{\gamma}_{2(i+1)} \ \cdots}\big)}\right) =y_iG_{ \big( {\cdots \ \ \tilde{\gamma}_{1(i+1)} \ \ \cdots \atop \cdots \ \tilde{\gamma}_{2(i+1)}-1 \ \cdots}\big)}$.

(3) If $\tilde{\gamma}_{1(i+1)}=0$, $\tilde{\gamma}_{2(i+1)}=1$ and $\tilde{\gamma}_{1(i+2)}>0$, then

$\phi^{-1}\left( G_{ \big( {\cdots \ 0 \ 0 \ \tilde{\gamma}_{1(i+2)} \ \cdots \atop \cdots \ 0 \ 1 \ \tilde{\gamma}_{2(i+2)} \ \cdots}\big)}\right) =y_iG_{ \big( {\cdots \ \tilde{\gamma}_{1(i+2)} \ \cdots \atop \cdots \ \tilde{\gamma}_{2(i+2)} \ \cdots}\big)}$.
\end{proof}

Then, we can construct a map $\Phi:\{X^{\tilde{\alpha}}F_{\beta}:|\beta|\geq 1\}\to\{G_{\tilde{\gamma}}\}$ that is defined by ``composing" $\phi$ with itself $(|\tilde{\alpha}|-1)$ times. By the above Lemma, we also have $\Phi$ is injective. For simplicity, we define $\phi^{-1}(G_{\tilde{\gamma}})$ (or $\Phi^{-1}(G_{\tilde{\gamma}})$) to be $X^{\tilde{\alpha}}G_{\tilde{\beta}}$ (or $X^{\tilde{\alpha}}F_{\beta}$) if $\phi(X^{\tilde{\alpha}}G_{\tilde{\beta}})=G_{\tilde{\gamma}}$ (or $\Phi(X^{\tilde{\alpha}}F_{\beta})=G_{\tilde{\gamma}}$ respectively).

\begin{Lemma}
In the expansion of $X^{\tilde{\alpha}}F_{\beta}$ in the $G$ basis using the rules above, the term $\Phi(X^{\tilde{\alpha}}F_{\beta})$ appears only once. In particular, it has coefficients $1$ or $-1$.
\end{Lemma}

\begin{proof}
We begin with the claim that if $\tilde{\mu}\neq\tilde{\nu}$, then $\phi(x_iG_{\tilde{\mu}})$ and $\phi(y_iG_{\tilde{\mu}})$ do not appear in the expansion of $x_iG_{\tilde{\nu}}$ and $y_iG_{\tilde{\nu}}$ respectively. 

Let $k$ be the smallest integer such that $(\tilde{\mu}_{k1},\tilde{\mu}_{k2})\neq(\tilde{\nu}_{k1},\tilde{\nu}_{k2})$. In rules (4.1), (4.2) and (4.3), for all $G_{\tilde{\gamma}}$ in the expansion of $x_iG_{\tilde{\mu}}$ or $y_iG_{\tilde{\mu}}$, the first $i-1$ columns of $\tilde{\gamma}$ is the same as that of $\tilde{\mu}$. Hence, the claim follows if $k<i$.

If $k=i$, and if we are multiplying $x_i$ using rules (4.1) or (4.2), then the claim holds because either the $i$-th or the $i+1$-th columns of $x_iG_{\tilde{\mu}}$ will be different from terms in expansions of $x_iG_{\tilde{\nu}}$. If we are multiplying by $y_i$, then note that if the $i-th$ column of $\mu$ is $(0,0)$, then $\mu_{(i+1)1}$ must be $0$ because otherwise, that means we multiplied an $x_i$ or $x_j$ or $y_j$ with $j>i$ before $y_i$, which violates our rule. And the same condition applies to $\nu$. With this restriction, it is easy to check that the claim holds.

If $k>i$, in both cases, if we choose any term in the expansion that is not $\phi(x_iG_{\tilde{\nu}})$ or $\phi(y_iG_{\tilde{\nu}})$, then the $i$ or $i+1$ column of its index must be different from that of $\phi(x_iG_{\tilde{\mu}})$ or $\phi(y_iG_{\tilde{\mu}})$. If we choose $\phi(x_iG_{\tilde{\nu}})$ or $\phi(y_iG_{\tilde{\nu}})$, we also have $\phi(x_iG_{\tilde{\mu}})\neq\phi(x_iG_{\tilde{\nu}})$ and $\phi(y_iG_{\tilde{\mu}})\phi(y_iG_{\tilde{\nu}})$ because $\mu\neq\nu$.

Since each term in the expansion of $X^{\tilde{\alpha}}F_{\beta}$ corresponds to a sequence of choice using rules (4.1), (4.2) or (43), if at some point, we choose a term that is different from the choice in $\Phi$, then a recursive use of the claim asserts that $\Phi(X^{\tilde{\alpha}}F_{\beta})$ will not appear again.
\end{proof}

We now define an order $(<_G)$ on the set of generalized bicompositions as follows

\begin{enumerate}
	\item If $\tilde{\alpha}$ and $\tilde{\beta}$ are bicompositions, then $\tilde{\alpha}<_G\tilde{\beta}$ if $\tilde{\alpha}<_{lex}\tilde{\beta}$.
	
	\item If $\tilde{\alpha}$ is a bicomposition and $\tilde{\beta}$ is not, then $\tilde{\alpha}<_G\tilde{\beta}$.
	
	\item If $\tilde{\alpha}=\tilde{\mu}\big( {0 \atop 0} \big)\alpha'$, $\tilde{\beta}=\tilde{\nu}\big( {0 \atop 0} \big)\beta'$ where $\alpha'$ and $\beta'$ are bicompositions, let $u=\ell(\tilde{\alpha})-\ell(\alpha')-1$, $v=\ell(\tilde{\alpha})-\ell(\beta')-1$, then $\tilde{\alpha}<_G\tilde{\beta}$ if
	
	\begin{enumerate}
		\item $u<v$, or
		\item $u=v$, $\alpha'_{11}>0$ and $\beta'_{11}=0$, or
		\item $u=v$, $\alpha'_{11}>0$, $\beta'_{11}>0$ (or $\alpha'_{11}=0$, $\beta'_{11}=0$) and $\overleftarrow{\phi}(G_{\tilde{\alpha}})<_G\overleftarrow{\phi}(G_{\tilde{\beta}})$ where we define $\overleftarrow{\phi}(G_{\tilde{\delta}})$ to be $\tilde{\gamma}$ if $\phi(x_iG_{\tilde{\gamma}})=G_{\tilde{\delta}}$ or $\phi(y_iG_{\tilde{\gamma}})=G_{\tilde{\delta}}$ for some $i$.
	\end{enumerate}
\end{enumerate}

\begin{Lemma}\label{lem:4.3}
The order defined above is a total order on the set of generalized bicompositions such that if $G_{\tilde{\gamma}}=\Phi(X^{\tilde{\alpha}}F_{\beta})$, then for all $G_{\tilde{\delta}}$ that appears in the expansion of $X^{\tilde{\alpha}}F_{\beta}$, we have $\tilde{\gamma}\geq_G\tilde{\delta}$.
\end{Lemma}

\begin{proof}
Clearly this is a total order. If $\tilde{\alpha}<\tilde{\beta}$ by (1) or (2), then $\tilde{\beta}$ cannot appear in the expansion of $\Phi^{-1}(\tilde{\alpha})=\tilde{\alpha}$.

If $\tilde{\alpha}<\tilde{\beta}$ by (3a), that means $\phi^{-1}(\tilde{\alpha})=x_{u+1}G_{\tilde{\gamma}}$ or $y_{u+1}G_{\tilde{\gamma}}$ for some $\tilde{\gamma}$. Hence, $\tilde{\beta}$ cannot appear in the expansion of $\Phi^{-1}(\tilde{\alpha})$ because $\tilde{\beta}_{(v+1)1}=\tilde{\beta}_{(v+1)2}=0$ cannot be created.

If $\tilde{\alpha}<\tilde{\beta}$ by (3b), that means $\phi^{-1}(\tilde{\alpha})=x_{u+1}G_{\tilde{\gamma}}$ for some $\tilde{\gamma}$. Hence, $\tilde{\beta}$ cannot appear in the expansion of $\Phi^{-1}(\tilde{\alpha})$ because it is not in that of $x_{u+1}G_{\tilde{\delta}}$ for any $\tilde{\delta}$.
\end{proof}

With this ordering, there is a unique leading $G_{\tilde{\delta}}$ for each expansion of $X^{\tilde{\alpha}}F_{\beta}$.

\begin{Theorem}\label{thm:4.4}
The set $A=\{G_{\tilde{\alpha}} \mid G_{\tilde{\alpha}}\notin \text{Img}(\Phi)\}$ forms a Hilbert basis for the quotient space R/I.
\end{Theorem}

\begin{proof}
For any polynomial $p$ in $R$, we write $p$ in terms of the $G$ basis with $<_G$ order. For each term $G_{\tilde{\alpha}}\in Img(\Phi)$, we subtract $p$ by $\Phi^{-1}(G_{\tilde{\alpha}})\in I$ and $G_{\tilde{\alpha}}$ is cancelled. If we repeat this process (possibly countably many times), we can express $p$ as a series of $A$.
\end{proof}

\section{Finitely many variables case}

In the case that there are finitely many variables, $R_n=\mathbb{Q}[x_1,\dots,x_n,y_1,\dots,y_n]$, the above constructions of $\DQsym(x_1,\dots,x_n,y_1,\dots,y_n)$, the $F$, $G$ bases and the ideal $I_n=<\DQsym^+(x_1,\dots,x_n,y_1,\dots,y_n)>$ remain the same by taking $x_i=y_i=0$ for $i>n$. In this case, $LM(G_{\tilde{\alpha}})=X^{\tilde{\alpha}}$ whenever $\ell(\tilde{\alpha})\leq n$ and hence $\{G_{\tilde{\alpha}}:\ell(\tilde{\alpha})\leq n\}$ spans $R_n$.

Let $R_n^{i,j}$ be the span of $\{X^{\tilde{\alpha}}:\ell(\tilde{\alpha})\leq n,\sum_k\tilde{\alpha}_{1k}=i, \sum_k\tilde{\alpha}_{2k}=j\}$. Since $I_n$ is bihomogeneous in $\mathbf{x}$ and $\mathbf{y}$, $\displaystyle I_n=\bigoplus_{i,j}I_n^{i,j}$ where $I_n^{i,j}=I_n\cap R_n^{i,j}$, and $\displaystyle R_n/I_n=\bigoplus_{i,j}V_n^{i,j}$ where $V_n^{i,j}=R_n/I_n\cap R_n^{i,j}$.

The Hilbert matrix corresponding to $R_n/I_n$ is the matrix $M_n(i,j)=\dim(V_n^{i-1,j-1})$.

The goal of this section is to compute the second column of the Hilbert matrix. The proof is slight generalization of the one in \cite{ABB}.

\begin{Lemma}\label{lem:5.1}
The set $\{G_{\tilde{\alpha}} \mid G_{\tilde{\alpha}}\notin Img(\Phi), \ell(\tilde{\alpha})\leq n\}$ spans the quotient $R_n/I_n$.
\end{Lemma}

\begin{proof}
Among all $\tilde{\alpha}$ such that $G_{\tilde{\alpha}}\in Img(\Phi)$, $\ell(\tilde{\alpha})\leq n$ and $G_{\tilde{\alpha}}$ cannot be reduced to $0$, let $\tilde{\beta}$ be the smallest one with respect to the $<_G$ order. Then, 
\begin{align*}
G_{\tilde{\beta}}&=G_{\tilde{\beta}}-\Phi^{-1}(G_{\tilde{\beta}})+\Phi^{-1}(G_{\tilde{\beta}})\\
&\equiv G_{\tilde{\beta}}-\Phi^{-1}(G_{\tilde{\beta}}) \mod I_n
\end{align*}
But since $G_{\tilde{\beta}}$ is the leading term in $\Phi^{-1}(G_{\tilde{\beta}})$, terms in $G_{\tilde{\beta}}-\Phi^{-1}(G_{\tilde{\beta}})$ are strictly smaller than $G_{\tilde{\beta}}$, and thus they reduce to $0$. This contradicts to our assumption on $\tilde{\beta}$.
\end{proof}

Let $B_n$ be the set of generalized bicompositions $\{\tilde{\alpha}\}$ such that $\displaystyle\sum_{i=1}^{k}(\tilde{\alpha}_{1i}+\tilde{\alpha}_{2i})<k$ for all $1\leq k\leq n$ and $\ell(\tilde{\alpha})\leq n$. Clearly from the definition of $G$ basis, if $\tilde{\alpha}\notin B_n$, then $G_{\tilde{\alpha}}\in I_n$. Therefore, the set $\{X^{\tilde{\alpha}}:\tilde{\alpha}\in B_n\}$ spans $R_n/I_n$, the proof is essentially the same as Lemma \ref{lem:5.1}. In particular, $X^{\tilde{\alpha}}\in I_n$ for all $|\tilde{\alpha}|\geq n$.

\begin{Lemma}\label{lem:5.2}
The set $\{X^{\tilde{\alpha}}F_{\beta}:\tilde{\alpha}\in B_n,|\beta|\geq 0\}$ spans $R_n$.
\end{Lemma}

\begin{proof}
We already have $\displaystyle X^{\tilde{\epsilon}}\equiv \sum_{\tilde{\alpha}\in B_n}X^{\tilde{\alpha}} \mod I_n$, which means $\displaystyle X^{\tilde{\epsilon}} = \sum_{\tilde{\alpha}\in B_n}X^{\tilde{\alpha}} + \sum_{|\beta|\geq 1}P_{\beta}F_{\beta}$ for some polynomial $P_{\beta}$. If we reduce each monomial $P_{\beta}$ using the above rule, and write the product of $F$ basis in terms of $F$ basis, the claim will be satisfied in a finite number of steps.
\end{proof}

For a generalized bicomposition $\tilde{\alpha}$ with $\ell(\tilde{\alpha})\leq n$, we define its reverse $\overline{\alpha}$ to be the generalized bicomposition such that $\overline{\alpha}_{1i}=\tilde{\alpha}_{1(n-i+1)}$ and $\overline{\alpha}_{2i}=\tilde{\alpha}_{2(n-i+1)}$ for all $1\leq i\leq n$.

We denote the set $\{X^{\tilde{\alpha}}:\overline{\alpha}\in B_n\}$ by $A_n$. The endomorphism of $R_n$ that sends $x_i$ to $x_{n-i+1}$ and $y_i$ to $y_{n-i+1}$ is clearly an algebra isomorphism that fixes $\DQsym(\mathbf{x},\mathbf{y})$, in fact, it sends $M_{\alpha}$ to $M_{\alpha'}$ where $\alpha'$ is the reversed bicomposition of $\alpha$. Therefore, by Lemma \ref{lem:5.2}, the set $\{X^{\tilde{\alpha}}F_{\beta}:\tilde{\alpha}\in A_n,|\beta|\geq 0\}$ spans $R_n$.

Hence, $I_n=\langle F_{\gamma}:|\gamma|\geq 1 \rangle$ is spanned by $\{X^{\tilde{\alpha}}F_{\beta}F_{\gamma}:\tilde{\alpha}\in A_n, |\beta|\geq0, |\gamma|\geq1\}$, which means it is spanned by $\{X^{\tilde{\alpha}}F_{\beta}:\tilde{\alpha}\in A_n, |\beta|\geq1\}$.

\begin{Lemma}\label{lem:5.3}
For $X^{\tilde{\alpha}}F_{\beta}\in R_n^{i,1}$ with $\tilde{\alpha}\in A_n$, $|\beta|\geq1$ and $|\tilde{\alpha}|+|\beta|<n$, let $G_{\tilde{\gamma}}=\Phi(X^{\tilde{\alpha}}F_{\beta})$, then $\ell(\tilde{\gamma})\leq n$.
\end{Lemma}

\begin{proof}
First, rules (4.1) and (4.2) increase the length by $1$ while (4.3) increase the length by $2$. Now, we need to track $\tilde{\gamma}_{\ell(\tilde{\gamma})}$. If $\tilde{\gamma}_{\ell(\tilde{\gamma})}$ comes from $\beta_{\ell(\beta)}$ and gets shifted, since we can use (4.3) at most once, we can make at most $|\tilde{\alpha}|+1$ steps to the right. Therefore, $\ell(\tilde{\gamma})\leq |\tilde{\alpha}|+1+\ell(\beta)\leq |\tilde{\alpha}|+1+|\beta|\leq n$.

If $\tilde{\gamma}_{\ell(\tilde{\gamma})}$ is $1$ which comes from multiplying $x_k$ or $y_k$ to $G_{\tilde{\epsilon}}$ with $k>\ell(\tilde{\epsilon})$, since $\tilde{\alpha}\in A_n$, we have $\sum_{i\geq k}(\tilde{\alpha}_{1i}+\tilde{\alpha}_{2i})<n-k+1$. In this process, we use rules (4.1) and (4.2) only and each increases the length by $1$. Therefore, $\tilde{\gamma}_{\ell(\tilde{\gamma})}$ can be shifted to at most position $k+n-k=n$.
\end{proof}

\begin{Corollary}
Let $M_n$ be the Hilbert matrix of $R_n/I_n$, then $M_n(n-1,2)=\displaystyle\frac{1}{n}\binom{2n-2}{n-1}$, $M_n(i,2)=\displaystyle\sum_{1\leq j\leq i,1\leq k\leq 2}M_{n-1}(j,k)$ for $1\leq i\leq n-2$, and $M_n(2,1)=0$ for $i\geq n$.
\end{Corollary}

\begin{proof}
Lemma \ref{lem:5.1} shows that $C_i=\{G_{\tilde{\alpha}}\in V_{n}^{i,1}:G_{\tilde{\alpha}}\notin Img(\Phi)\}$ spans $V_n^{i,1}$. Suppose there is a linear dependence $\displaystyle P=\sum_{G_{\tilde{\alpha}}\in C_i}a_{\tilde{\alpha}}G_{\tilde{\alpha}}\in I_n^{i,1}$. Since $I_n^{i,1}$ is spanned by $D=\{X^{\tilde{\alpha}}F_{\beta}\in R_n^{i,1}:\tilde{\alpha}\in A_n, |\beta|\geq1\}$, we have $P=\displaystyle\sum_{X^{\tilde{\alpha}}F_{\beta}\in D}b_{\tilde{\alpha}\beta}X^{\tilde{\alpha}}F_{\beta}$. This means the leading term of $P$ when we expand in $G$ basis is some $G_{\tilde{\gamma}}$ such that $\tilde{\gamma}\in Img(\Phi)$ and by Lemma \ref{lem:5.3} $\ell(\tilde{\gamma})\leq n$, which is absurd. Therefore, $C_i$ is a linear basis for $V_n^{i,1}$.

Now, $M_n(i,1)=\dim V_n^{i-1,1}=|C_{i-1}|$. Let $G_{\tilde{\gamma}}\in V_{n}^{i,1}$ and $k$ be the unique number that $\tilde{\gamma}_{k2}=1$. First, from definition of $G$, $\tilde{\gamma}\notin B_n$ implies $G_{\tilde{\gamma}}\in I_n$ and $G_{\tilde{\gamma}}\in Img(\Phi)$.

If $i=n-1$, then $|\tilde{\gamma}|=n-1$. If $k<\ell(\tilde{\gamma})$, since $\displaystyle\sum_{j=k+1}^n\tilde{\gamma}_{1j}\geq n-k$, we will be using rules (4.3) when applying $\phi^{-1}$. This reduces the length by $2$ while the size by $1$, which means $G_{\tilde{\gamma}}\in Img(\Phi)$. If $k=\ell(\tilde{\gamma})$, we only use rules (4.1) and (4.2) when applying $\phi^{-1}$. In this case, $G_{\tilde{\gamma}}\notin Img(\Phi)$ whenever $\tilde{\gamma}\in B_n$. Therefore, $|C_{n-2}|$ is the Catalan number $\displaystyle\frac{1}{n}\binom{2n-2}{n-1}$.

If $1\leq i\leq n-2$, $|\tilde{\gamma}|\leq n-2$. From the definition of $\phi$, $G_{\tilde{\gamma}}\notin Img(\Phi)$ if and only if $G_{\big( {\tilde{\gamma}_{11} \ \cdots \ \tilde{\gamma}_{1(n-1)} \atop \tilde{\gamma}_{21} \ \cdots \ \tilde{\gamma}_{2(n-1)}} \big)}\in V_{n-1}^{j,k}\setminus Img(\Phi)$ for some $1\leq j\leq i,1\leq k\leq 2$. Therefore, $M_n(i,2)=\displaystyle\sum_{1\leq j\leq i,1\leq k\leq 2}M_{n-1}(j,k)$ for $1\leq i\leq n-2$.
\end{proof}

By the symmetry $M_n(a,b)=M_n(b,a)$, we obtain the first to rows of the Hilbert matrix, namely $M_n(2,n-1)=\displaystyle\frac{1}{n}\binom{2n-2}{n-1}$, $M_n(2,i)=\displaystyle\sum_{1\leq j\leq i,1\leq k\leq 2}M_{n-1}(k,j)$ for $1\leq i\leq n-2$, and $M_n(2,i)=0$ for $i\geq n$.

This method can be applied directly to some other terms. To be more specific, for $2i+j\leq n$,  the set $\{G_{\tilde{\alpha}} \mid G_{\tilde{\alpha}}\notin Img(\Phi), \ell(\tilde{\alpha})\leq n\}$ is a linear basis in $V_n^{i,j}$. Therefore, the formula for each column stabilizes when the number of variables is large enough. However, it fail in other terms and this set is not a linear basis in general.

\end{document}